\documentclass[smallextended,amstex]{svjour2}     
\smartqed  
\usepackage{graphicx}
\usepackage{amsmath}
\usepackage{amssymb}
\usepackage{color}
\begin{document}

\title{On Rank Driven Dynamical Systems
}


\author{J.J.P. Veerman         \and
        F.J. Prieto
}

\institute{J.J.P. Veerman \at
              Fariborz Maseeh Dept of Math and Stat\\ Portland State Univ.\\ Portland (OR), USA \\
              \email{veerman@pdx.edu}             \\
           \and
           F.J. Prieto \at
              Dept of Statistics \\ Univ Carlos III \\ Madrid, Spain \\
              \email{franciscojavier.prieto@uc3m.es}
}

\date{Received: date / Accepted: date}

\maketitle

\begin{abstract}
We investigate a class of models related to the Bak-Sneppen model, initially proposed to study evolution. The BS model is extremely simple and yet captures some forms of ``complex behavior'' such as self-organized criticality that is often observed in physical and biological systems.

In this model, random fitnesses in $[0,1]$ are associated to agents located at the vertices of a graph $G$. Their fitnesses are ranked from worst (0) to best (1). At every time-step the agent with the worst fitness and some others \emph{with a priori given rank probabilities} are replaced by new agents with random fitnesses. We consider two cases: The \emph{exogenous case} where the new fitnesses are taken from an a priori fixed distribution, and the \emph{endogenous case} where the new fitnesses are taken from the current distribution as it evolves.

We approximate the dynamics by making a simplifying independence assumption. We use Order Statistics and Dynamical Systems to define a \emph{rank-driven dynamical system} that approximates the evolution of the \emph{distribution} of the fitnesses in these rank-driven models, as well as in the Bak-Sneppen model. For this simplified model we can find the limiting marginal distribution as a function of the initial conditions. Agreement with experimental results of the BS model is excellent.

\end{abstract}

\section{Introduction}
\label{chap:introduction}

The publication of the paper by Bak and Sneppen \cite{BS} gave rise to a significant number of contributions in the literature both on evolution modeling and in related fields, regarding the analysis and application of this model, and other similar models. This interest has been explained on the basis of the simplicity of the model and its ability to capture some forms of complex behavior such as ``self-organized criticality'', \cite{Jen}, a property that is often observed in physical and biological systems. These are of interest in the modeling of evolution processes but also of other processes. In particular, the model shows power-law behavior in the distribution of some of its performance measures, while not requiring the fine-tuning of any parameters to achieve this criticality, thus it is considered to be ``self-organized''.

The study of the theoretical properties of the model, and in general of related models presenting this self-organized criticality property, has been based mostly on the analysis of ``avalanches'' associated to its evolution. These avalanches are defined as epochs in the process between times when the fitnesses of all agents are above a certain threshold. In particular, in Paczuski et al.\ \cite{PMB} some results were obtained on the evolution of the gap between avalanches and the span of an avalanche, for a class of self-organized models (``extremal'' models) that included Bak-Sneppen.

In spite of the simplicity of the Bak-Sneppen (BS) model and the significant literature devoted to its study, its properties have proven to be quite difficult to analyze in detail, and many of them are known mainly from simulation results. In Meester and Znamenski \cite{MZ2} and \cite{MZ3}, and Meester et al.\ \cite{MGW}, for example, the non-triviality of the marginal distributions was proved and some parameters associated to the limiting distributions were studied, based on a characterization of the behavior of avalanches. Other results establishing different properties of the limiting behavior of these avalanches in the BS model can be found in Maslov \cite{Mas} and Tabelow \cite{Tab}.

A more detailed understanding of these models can be obtained from the consideration of simplified versions, which would still preserve some of the interesting characteristics of BS while being much more amenable to analysis. One of the first modifications to be considered was the ``mean-field'' variant of BS, that is, the case when the replaced agents were chosen as the one with the minimum fitness and a fixed number of other agents selected at random. This model was studied in de Boer et al.\ \cite{BDF}, where it was shown that power laws were still obtained for the distribution of avalanche durations. Meester and Znamenski \cite{MZ1} present another model where fitnesses were required to take the values either 0 or 1, and show that it exhibits non-trivial behavior.

Another promising and more complex approach to approximate the models of interest has been proposed by Grinfeld et al.\ \cite{GKW1}. They introduced and adapted so-called ``rank-driven'' processes, that is, processes where the agents to be replaced are defined in terms of the order statistics of the current population distribution. These processes have been described in detail, and their limiting properties have been characterized for some cases, in Grinfeld et al.\ \cite{GKW2}. In particular, they showed that if the replacements were chosen from a uniform distribution, the limiting distribution for appropriate rank-driven processes had a structure similar to that observed in simulations of a BS process. We are interested in considering situations that extend these cases to more general settings.

There is no perfect equivalence between a Bak-Sneppen process and rank-driven processes. Nevertheless, the results available show that the limiting marginal distributions of these processes provide reasonable approximations to study the limiting behavior of the Bak-Sneppen process and some of its extensions. The usefulness of these processes as approximations for other more complex ones, as well as their flexibility to generate different types of limiting behaviors, provide our main motivation to study this class of processes.

We introduce a simplifying hypothesis, namely that our observations are statistically independent. We then use the resulting approximate theory to characterize the limiting distribution of rank-driven processes for any continuous replacement distributions, in a manner that is simple and computationally tractable. The result is a new type of dimensional dynamical system acting on the space of cumulative distribution functions, and whose solutions are accessible through analysis. We will denote these systems by \emph{rank-driven dynamical systems}.

These same procedures are also used to analyze the behavior of the system in a case that has not been previously considered in the literature, namely when the replacements are taken from the same population (the ``endogenous'' case). Also here the limiting measures are easily computable. An interesting dichotomy occurs. In the endogenous case the limiting measure is singular with respect to Lebesgue. In the case discussed above where the replacements are taken from an unchanging distribution --- we refer to this situation as the \emph{exogenous} case, the limiting measures for many reasonable initial conditions are absolutely continuous, though in some cases the underlying exact solution may still be singular (see Section \ref{chap:discussion}). Thus, endogenous evolution may optimize fitness if the conditions are right, but all agents tend to become identical (no diversity). Exogenous evolution on the other hand can also improve fitness, but often diversity is retained: a nontrivial \emph{range} of fitnesses is preserved.

The paper is organized as follows: In Section \ref{chap:interlude} we present the rank-driven processes of interest, and we enumerate some well-known results on the distribution of order statistics. Section \ref{chap:zero} characterizes the temporal evolution of the process under endogenous and exogenous replacement schemes. It also introduces and justifies the main characterization results for both of these cases. In Section \ref{chap:two} we give some examples of these dynamics and show that we get behaviors similar to those observed under the BS model. Section \ref{chap:asympt} presents a precise characterization of the limiting measure (in the limit for many agents) for exogenous replacement schemes similar to, but much more general then, the one used in the BS model. In the last section we compare our approximation to the known exact solution for a particular system.

\section{Rank-driven processes and order statistics}
\label{chap:interlude}

In the Bak-Sneppen model, the lowest value plus its two immediate neighbors in a given graph $G$ are replaced by random values from a known distribution $N$. This procedure introduces a ``geometric'' dependence that greatly complicates its analysis. Simulation results indicate that after many iterations the neighbors of the worst performer consist of (on average) one very bad performer and one ($N$-) average performer. Indeed, it appears very hard to prove these correlations (see for example \cite{GKW1}). It seems therefore reasonable to approximate the BS model by a rank-driven model whose parameters can be chosen to obtain a model that is much more tractable and provides a close approximation.

We now introduce a (simplified) formal model for this system, and its associated notation. Let $R$ be a distribution function for a distribution supported on $[0,1]$. Let ${\cal G} = {\cal G} (R; n) = \{G_1, \ldots , G_n\}$ be a sample of $n$ (almost surely distinct) random variables with marginal distribution $R$, i.e., $P[G_i \leq x] = R(x)$ for each $i$. Denote the corresponding order statistics $G_{(1)} \leq \cdots \leq G_{(n)}$, and the marginal distribution function of the $i$-th order statistic by $\Theta_{i:n}(R) (x) = P[G_{(i)} \leq x]$, where the notation $\Theta_{i:n} (R)$ is motivated by our interest in the study of the properties of $R$.

\begin{lemma}
Consider the $k$-out-of-$n$ order-statistics $G_{(k)}$ from a set of identically distributed observations. Their distributions $\Theta_{k:n} (R)$ satisfy:
\begin{equation*}
 R = \frac1n \sum_{i=1}^n\, \Theta_{i:n} (R) .
\end{equation*}
 \label{lem:order-stats}
\end{lemma}

\begin{proof}
It holds that
\[
  \sum_{i=1}^n 1\{{G_{(i)} \leq x}\} = \sum_{i=1}^n 1\{{G_i \leq x}\},
\]
and taking expectations,
\[
  \frac1n \sum_{i=1}^n \Theta_{i:n} (R) (x) = R(x) .
\]
\end{proof}
Let ${\cal E}(x) = {\cal E}(G)(x)$ denote the empirical distribution of $G$:
\[
   {\cal E} (x) = \frac1n \sum_{i=1}^n 1\{{G_{i} \leq x}\} = \frac1n \sum_{i=1}^n 1\{{G_{(i)} \leq x}\} = \frac1n \max \{ i : G_{(i)} \leq x \}
\]
Then,
\[
   {\mathbb E} [{\cal E} (x) ] = R(x) ,
\]
and $R$ is also the one-dimensional marginal distribution of a randomly chosen member of ${\cal G}$. The characterization of this distribution is the object of interest of this paper.

Assume now a random index $I$, uniformly distributed on $\{ 1,\ldots,N \}$. It holds that
\[
   {\mathbb E} [{\mathbb P} [ G_{(I)} \leq x] ] = \frac1n \sum_{i=1}^n \Theta_{i:n} (R) (x) = R(x) .
\]

On the model described above, we will conduct the following stochastic ``update'': Let $\xi_1,\ldots , \xi_n$ and $Y_1, \ldots , Y_n$ be independent (and independent of ${\cal G}$), with ${\cal P}[\xi_i = 1] = 1 - {\cal P}[\xi_i = 0] = \alpha_i \in [0,1]$, and ${\cal P}[Y_i \leq x] = N(x)$. Construct ${\cal G}'$ by, for each $i$, selecting $Y_i$ if $\xi_i =1$ or $G_{(i)}$ if $\xi_i = 0$:
\[
  {\cal G}' = \{ G_{(i)} : \xi_i = 0\} \cup \{ Y_i : \xi_i = 1 \} .
\]
Then the ``updated'' empirical distribution corresponding to ${\cal G}'$ is
\[
  {\cal E}'(x) = {\cal E} ({\cal G}') (x) = \frac1n \sum_{i=1}^n \left( \xi_i 1\{ Y_i \leq x \} + (1-\xi_i) 1\{ G_{(i)} \leq x \}\right) .
\]
Using Lemma \ref{lem:order-stats} it follows that
\begin{eqnarray*}
 {\mathbb E} [{\cal E} (x) ] & = & \frac1n {\mathbb E} \sum_{i=1}^n \xi_i 1\{ Y_i \leq x \} + \frac1n {\mathbb E} \sum_{i=1}^n 1\{ G_{(i)} \leq x \} - \frac1n {\mathbb E} \sum_{i=1}^n \xi_i 1\{ G_{(i)} \leq x \} \\
 & = & \frac1n \sum_{i=1}^n \alpha_i N(x) + R(x) - \frac1n \sum_{i=1}^n \alpha_i \Theta_{i:n} (R) (x) .
\end{eqnarray*}
This is the one-dimensional marginal of a randomly chosen member of the new sample ${\cal G}'$. The iteration of this update will be the basis to define the dynamics of our rank-driven dynamical system.

At this point we introduce a distinction between exogenous and endogenous evolution in a rank-driven dynamical system. In the first case the cdf $N(x)$ represents individuals from an outside (\emph{exogenous}) distribution replacing disappearing individuals. Thus in \emph{exogenous} evolution we consider $N(x)$ to be fixed. One can say that in this case $N(x)$ is the distribution that ``drives'' the dynamics. In contrast, \emph{endogenous} evolution would replace the selected fitnesses with new fitnesses from the same distribution. Those new individuals are thus taken from the same (marginal) distribution as the disappearing individuals. Thus in this case the driving measure is $R$ itself.

\begin{definition}
The following are the equations governing exogenous and endogenous evolution ($n\in \mathbb{N}$ fixed) for a rank-driven system
\begin{eqnarray}
\label{eq:defexo}
  \phi_{exo}(R)(x) & = & R(x) - \dfrac 1n \sum_{i=1}^n\,\alpha_i \,\Theta_{i:n} (R)(x) + \dfrac 1n \sum_{i=1}^n\,\alpha_i\,N(x) \\
\label{eq:defendo}
  \phi_{endo}(R)(x) & = & R(x) - \dfrac 1n \sum_{i=1}^n\,\alpha_i \,\Theta_{i:n} (R)(x) + \dfrac 1n \sum_{i=1}^n\,\alpha_i\,R(x)
\end{eqnarray}
where for all $i$, $\alpha_i\in[0,1]$.
\label{defn:exo-endo}
\end{definition}
Note that from Lemma \ref{lem:order-stats}, $\phi_{exo}(R)$ in Definition \ref{defn:exo-endo} can be written as
\[
  \phi_{exo}(R) = \dfrac1n \sum_{i=1}^n\,\left( (1-\alpha_i)\Theta_{i:n} (R) + \alpha_i N \right) .
\]
Since $\alpha_i \in [0,1]$ and $N$ and $\Theta_{i:n} (R)$ are distributions, it follows that $\phi_{exo} (R)$ is also a distribution. As a consequence, both equations map the space $M \equiv M^1([0,1])$ of cumulative distribution functions on $[0,1]$ to itself.

The intuition behind this exogenous dynamic model is that at each iteration we take out some values, namely for each $i$ we take out (on the average) $\alpha_i$ times the $i$-th lowest value, and replace those by random values from the cdf $N$. Our goal is the study of the marginal cumulative stationary distribution functions (with support in $[0,1]$) of the values $G$. This study will be based on the properties of order statistics, and in particular on the characterization of their distribution functions.

To conduct this analysis we will impose the additional assumption that, at each stage, the $G_1, \ldots , G_n$ are independent draws from their marginal distribution. Under this independence assumption, we now present several well-known results related to the behavior of order statistics and their distributions. Most of these results are standard and can be found in many basic references, such as for example \cite{DN}. These are presented without proof.

\begin{theorem}
Given a (cumulative) distribution (function) $R(x)$, for the $k$-out-of-$n$ order-statistics $G_{(k)}$ from an independent set of trials we have distribution functions $\Theta_{k:n} (R)$,
\[
  \Theta_{k:n} (R) = \sum_{i=k}^n\,\binom{n}{i}\,R^i(1-R)^{n-i}= 1- \sum_{i=0}^{k-1}\, \binom{n}{i}\,R^i(1-R)^{n-i}
\]
\label{theo:order-statistics}
\end{theorem}

\begin{lemma}
The partial of $\Theta_{k:n} (R)$ with respect to $R$ is non-negative, and is strictly positive if $R\in (0,1)$. Furthermore we have that
\begin{eqnarray*}
\partial_R \Theta_{1:n} (R) |_{R=0}=n \quad &\mbox{ and }& \quad \mbox{ if }\; k>1\; :\;\partial_R \Theta_{k:n} (R) |_{R=0}=0\\
\partial_R \Theta_{n:n} (R) |_{R=1}=n \quad &\mbox{ and }& \quad \mbox{ if }\; k<n\; :\;\partial_R \Theta_{k:n} (R)|_{R=1}=0
\end{eqnarray*}
\label{lem:derivatives}
\end{lemma}

\begin{proof}
This result follows from the expressions for $\Theta_{k:n} (R)$ in Theorem \ref{theo:order-statistics} being differentiable with respect to $R$, and satisfying
\[
  \partial_R \Theta_{k:n} (R) = n \binom{n-1}{k-1}R^{k-1}(1-R)^{n-k} .
\]
\end{proof}

\begin{lemma}
Let $\alpha_i$ as in Definition \ref{defn:exo-endo}. Then $\partial_R \left(\frac 1n\,\sum_i\,\alpha_i \Theta_{i:n} (R) \right) (x) \in [0,1]$. It can only be equal to one if $R(x)$ is 0 or 1.
\label{lem:derivatives2}
\end{lemma}

\begin{proof}
We note that $\partial_R\,R=1$ which, by Lemma \ref{lem:order-stats}, equals:
\begin{equation*}
  \partial_R \left(\frac 1n\,\sum_i\,\Theta_{i:n} \right)
  = \partial_R \left(\frac 1n\, \sum_i\,\alpha_i \Theta_{i:n} \right) + \partial_R \left(\frac 1n\,\sum_i\,(1-\alpha_i) \Theta_{i:n}\right)
\end{equation*}
Since all the coefficients are nonnegative, we have that by the previous Lemma each of these two terms is nonnegative with the extremal case only possible if $R=0$ or $R=1$.
\end{proof}

It is amusing to illustrate these concepts with a simple coin-toss example. As usual we denote the Heaviside step function with the step at $r$ by $H(x-r)$ and its distributional derivative the delta function by $\delta(x-r)$. Suppose
\begin{equation*}
  R(x) = \frac 12 \left( H(x) + H(x-1)\right) \mbox{ or } \rho(x)= \frac 12 \left( \delta(x) + \delta(x-1)\right).
\end{equation*}
This corresponds to throwing head (=0) or tail (=1) with equal probability. One easily checks that for a set of two independent trials one has a probability of $\frac 34$ that the lowest throw is a 0 and the same probability that the highest throw is a 1. Indeed the formulae of Theorem \ref{theo:order-statistics} give us that
\begin{equation*}
  \Theta_{1:2} = 2R - R^2 \quad \mbox{ and } \quad \Theta_{2:2}= R^2
\end{equation*}
Note that these distributions sum to $2R$ as in Lemma \ref{lem:order-stats} and that their derivatives with respect to $R$ satisfy Lemma \ref{lem:derivatives}. It is also straightforward to check from this that
\begin{equation*}
  \Theta_{1:2}(R)(x) = \frac 34 H(x) + \frac 14 H(x-1) \quad \mbox{ and } \quad
  \Theta_{2:2}(R)(x) = \frac 14 H(x) + \frac 34 H(x-1)
\end{equation*}
These expression illustrate the fact that even though the $\Theta_{1:n}(R)(x)$ may very well be singular, the $\Theta_{1:2}$ are smooth polynomials in $R$.



\section{Rank Driven Dynamics}
\label{chap:zero}

As we mentioned above, in order to work with a manageable representation of the distribution functions of the order statistics, we will study the system assuming that the different components of $G$ were independent.

Note that in general, even if we start with $n$ i.i.d. random variables, this independence property will not be preserved by our dynamical systems. Thus, this assumption would seem to be quite strong, but we believe it is not entirely unjustified. For example, it has been observed that it appears that in the limit as $N \rightarrow \infty$ the fitnesses behave as if they are independent \cite{MS}, and that at the end of an avalanche the fitnesses of the elements affected by it are independent \cite{MZ2}. Also, it seems reasonable to assume that the stationary distribution for this process is at least interchangeable, implying that its marginal distribution could be recovered from the fixed-point of the dynamical system under the independence assumption. One interesting insight resulting from our assumption and our ability to recover the limiting behavior of the marginal distributions in the BS model, see Section \ref{chap:asympt}, is that the dependence structure of these models does not seem to impact this limiting behavior.
Note that we impose the condition that the $\xi_i$ and $Y_i$ used to define our iteration are independent, and independent of ${\cal G}$. This condition corresponds to the structure used in the model defined in \cite{GKW2}.

We now show that under this assumption the measures $\phi^\ell(R)(x)$ converge.

\begin{theorem}
Let $\phi_{exo}$ and $\phi_{endo}$ be the dynamical systems on $M$ as given in Definition \ref{defn:exo-endo}. Assume that at each step the distributions of the fitnesses are independent. Then there are unique $R_{exo}^*=\Gamma_{exo}^*(N)$ and $R_{endo}^* = \Gamma^*_{endo} (R_0)$ in $M$ such that for every starting measure $R_0(x)$ and each $x\in[0,1]$ we have
\begin{eqnarray}
\label{eq:exo}
  \lim_{n\rightarrow\infty} \phi_{exo}^n(R(x))&=& (\Gamma_{exo}^*\circ N)(x) \\
\label{eq:endo}
  \lim_{n\rightarrow\infty} \phi_{endo}^n(R(x))&=& (\Gamma_{endo}^*\circ R_0)(x)
\end{eqnarray}
Furthermore $\Gamma_{exo}^*$ is strictly increasing and $C^\infty$ on $(0,1)$ while $R^*_{endo} = \Gamma_{endo}^*(R_0)$ is the cdf corresponding to a weighted sum of finitely many delta distributions.
\label{theo:theo1}
\end{theorem}

\begin{figure}[pbth]
\center
\includegraphics[width=2.2in]{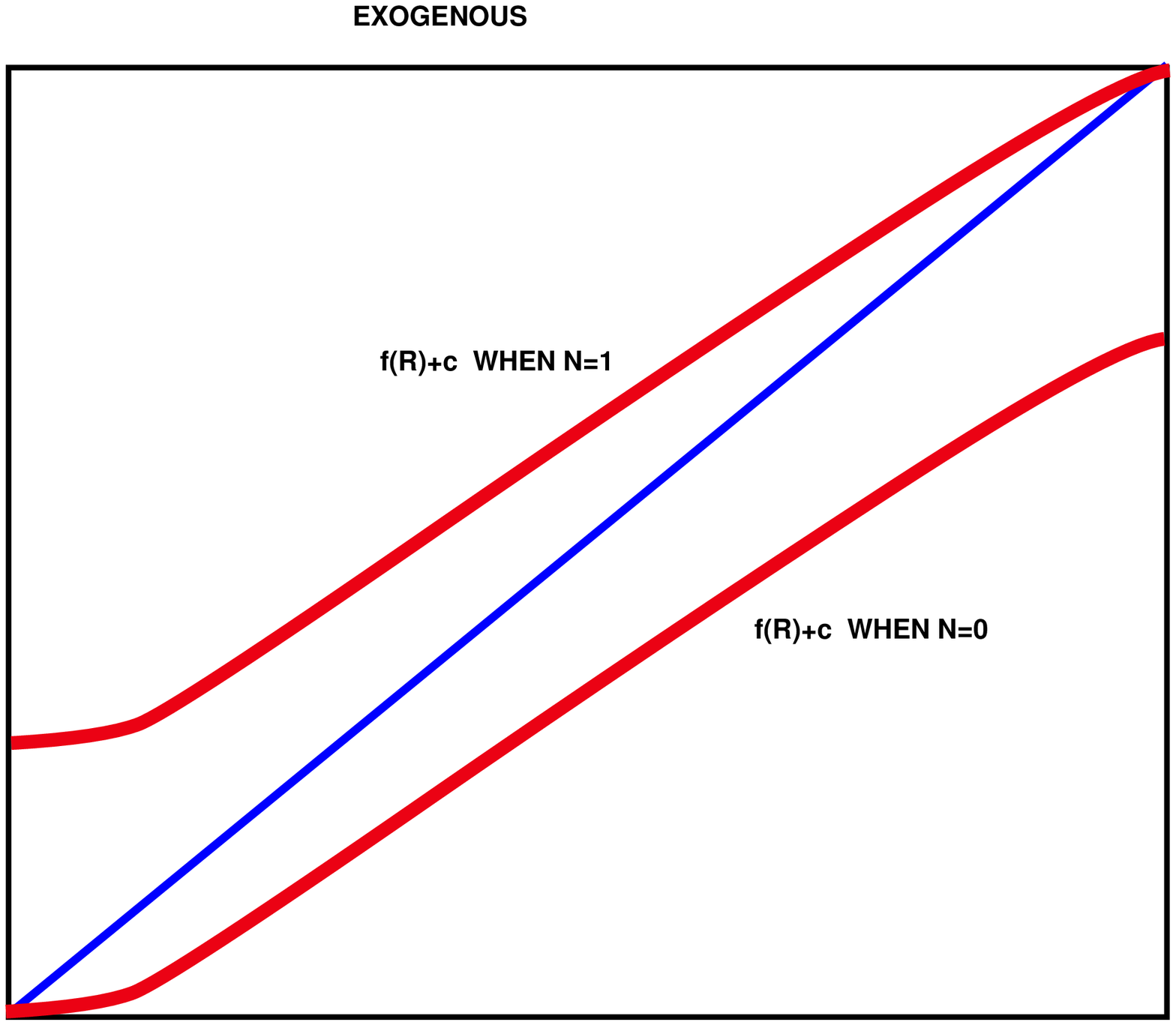}
\includegraphics[width=2.2in]{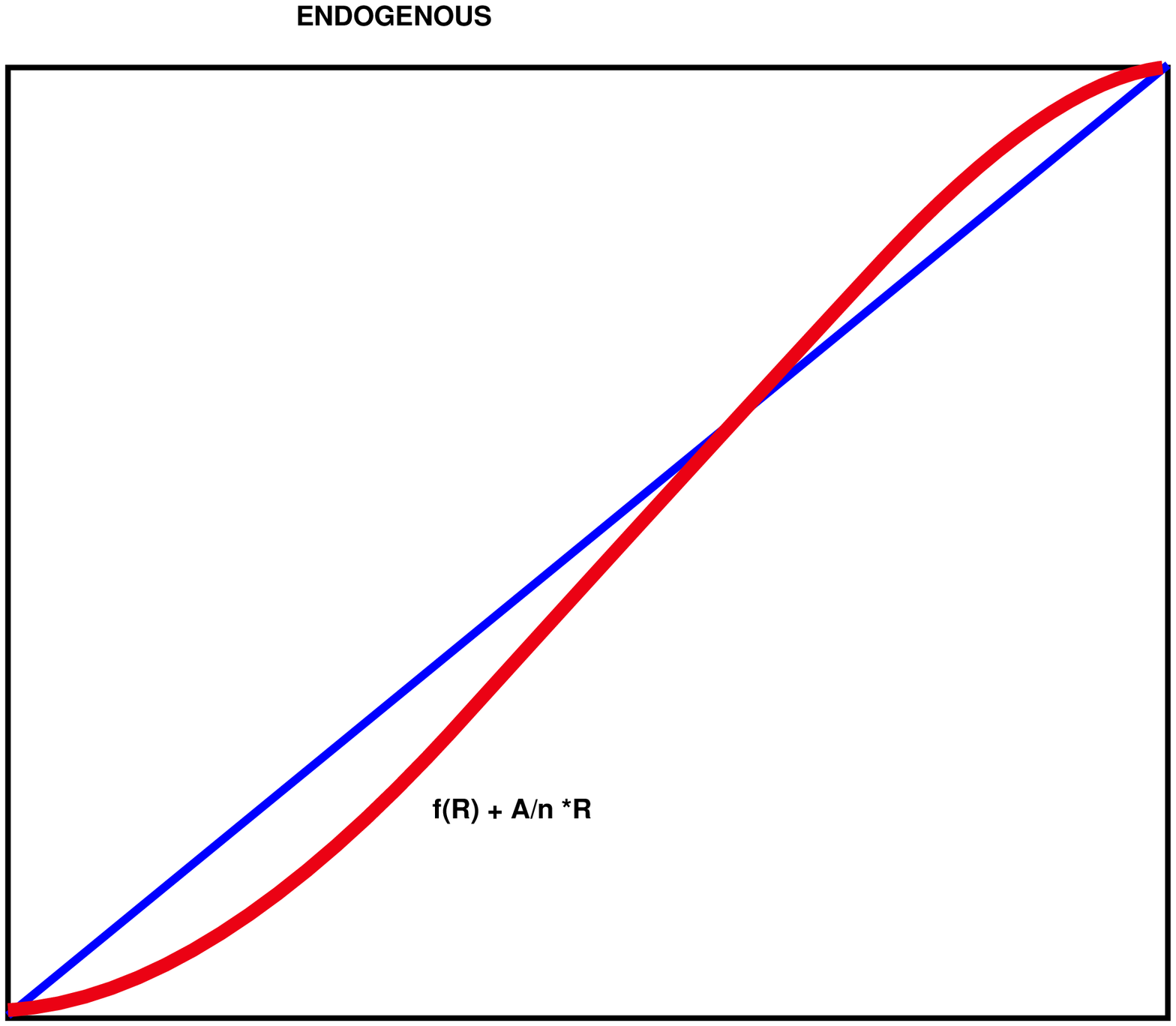}
\caption{\emph{In the first figure the 1-parameter family of attracting dynamical systems induced by $\phi_{exo}$; in this case there is always a unique attracting fixed point. In the second, the dynamical system induced by $\phi_{endo}$; now repelling fixed points are possible as illustrated in the figure.}}
\label{spec8}
\end{figure}

\begin{proof} From Definition \ref{defn:exo-endo} we see that the dynamical system on $M$ given by $R\rightarrow \phi(R)$ induces an associated dynamical system $\tilde \phi$ from $I=[0,1]$ to itself, via the commuting diagram below ($Id$ is the identity):
\begin{equation*}
\begin{array}{ccccc}
& & \tilde\phi &&\\
&I&\rightarrow&I& \\
R& \uparrow &&\uparrow & \phi R\\
&I&\rightarrow&I& \\
& & Id &&
\end{array}
\end{equation*}
From now on we replace the study of $\phi$ by that of $\tilde \phi$ and since no confusion is likely to arise we drop the tilde from the notation. Letting $q \equiv \sum_i \alpha_i$, we get:
\begin{equation}
\begin{array}{ccc}
  \phi_{exo}(R) & = & R - \dfrac 1n \sum_{i=1}^n\,\alpha_i \,\Theta_{i:n}(R) + \dfrac qn \,N \\
  \phi_{endo}(R)&=& \left(1+\frac qn\right) R - \dfrac 1n \sum_{i=1}^n\,\alpha_i \,\Theta_{i:n}(R)
\end{array}
\label{eqn:fixedpoints}
\end{equation}
In both cases the induced map $\phi$ is monotone (because it maps cumulative distribution functions to cumulative distribution functions). From Theorem \ref{theo:order-statistics} it is polynomial and not equal to the identity (assuming $q < n$). Thus it can only have finitely many fixed points.

For the exogenous case we obtain a 1-parameter family of dynamical systems of the form:
\begin{equation*}
  \phi_{exo}(R)=f(R)+\dfrac qn \,N
\end{equation*}
see Figure \ref{spec8}. By Lemma \ref{lem:derivatives2} we have that $0<|f'(R)|<1$ (except possibly at the endpoints). Thus, by the Mean Value Theorem, for value of $N$, there is a unique fixed point $R^* = \Gamma^* (N)$ which is a global attractor. Here $\Gamma^*$ denotes the mapping from $M$ to $M$ that captures the dependency of the fixed-point distribution $R^*$ on the driving distribution $N$. It is easily seen that $f(0)=0$ and $f(1)=1-q/n$ and therefore that $R^*(0)=0$ and $R^*(1)=1$. We can use the Implicit Function Theorem to get $C^\infty$ dependence of $R^*$ on $N$. The dependence on $x$ is given by $R^*(x) = \Gamma^*(N(x))$.

In the endogenous case we obtain a dynamical system of the form:
\begin{equation*}
  \phi_{endo}(R)=f(R)+\frac qn R
\end{equation*}
Since as before $f(0)=0$ and $f(1)=1-q/n$, we conclude that $\phi(0)=0$ and $\phi(1)=1$ (see Figure \ref{spec8}). We already know that $\phi$ has finitely many fixed points. Since $\phi$ is nondecreasing all fixed points with slope less than 1 must be attracting, and those with slope greater than 1 are repelling. Thus for any starting distribution $R_0$, we have that $\Gamma^* (R_0)=\lim_{n\rightarrow \infty} \phi^n(R_0)$ assumes finitely many values. The limiting distribution is given by $R^*(x) = \Gamma^*(R_0(x))$.
\end{proof}

{\bf Remark:} It follows from the last part of the proof that $\Gamma^* (R_0)$ is constant in a neighborhood of an attracting fixed point. In fact it is constant in the entire basin of attraction. Thus, it has a discontinuity on a boundary of a basin of attraction. This gives rise to the somewhat counter-intuitive fact that the limiting measure is concentrated on the \emph{repelling} fixed points (or the marginal ones).

We will later need a result that is immediately implied by Theorem \ref{theo:theo1} and Equation \ref{eqn:fixedpoints} in its proof.

\begin{lemma}
The fixed points of $\phi_{exo}$ and $\phi_{endo}$ are determined by:
\begin{eqnarray*}
\mbox{ For }\phi_{exo}\quad &:& \quad \sum_{i=1}^n\,\alpha_i \Theta_{i:n} (R) = q N \\
\mbox{ For }\phi_{endo}\quad &:& \quad \sum_{i=1}^n\,\alpha_i \Theta_{i:n} (R) = q R
\end{eqnarray*}
\label{lem:fixedpoints}
\end{lemma}

\section{Examples of Rank Driven Dynamics}
\label{chap:two}

In this section we illustrate some of the results that can be obtained from the application of the preceding Theorems. We start with a very simple example: an array of $n>1$ numbers where we only replace the lowest ranked member, that is, $\alpha_1 = 1$ and $\alpha_k = 0$ for $k \geq 2$. We study both the endogenous and exogenous cases.
\begin{eqnarray*}
  \phi_{exo}(R) & = & R - \dfrac 1n \Theta_{1:n} (R) + \dfrac 1n N \\
  \phi_{endo}(R) & = & (1+\dfrac1n)R - \dfrac 1n \Theta_{1:n} (R)
\end{eqnarray*}
Since we have
\begin{equation*}
  R_{1:n} = 1-(1-R)^n
\end{equation*}
we get
\begin{eqnarray*}
  \phi_{exo}(R)&=& R -\dfrac 1n (1-(1-R)^n) + \dfrac 1n N\\
  \phi_{endo}(R)&=& (1+\dfrac1n)R - \dfrac 1n (1-(1-R)^n)
\end{eqnarray*}
We solve in both cases for the fixed points (see Figure \ref{fig1}):
\begin{equation}
\left\{
  \begin{array}{lcl}
  {\rm exogenous} & : & \quad 1-(1-R)^n = N \; \Rightarrow \; \Gamma_{exo}^*(N) = 1-(1-N)^{1/n} \\[.12in]
  {\rm endogenous} & : & \quad R-1+(1-R)^n=0 \; \Rightarrow \; R=0 \mbox{ or } R=1\; \\
  & & \quad \Rightarrow \; \Gamma_{endo}^*(R_0)=H_1(R_0)
  \end{array}
  \right.
\label{eq:explicit}
\end{equation}
where $H_\alpha$ is the Heaviside function with jump at $\alpha$. In the exogenous case this gives a smooth function (not differentiable at $N=1$). In the endogenous case, $R_0=0$ is an attracting fixed point and $R_0=1$ a repelling one. Hence the limiting measure is concentrated in $R_0=1$ (see the remark at the end of Section \ref{chap:zero}). From Theorem \ref{theo:theo1} we see that if we choose $N(x)=R_0(x)=x$ then equation \ref{eq:explicit} gives the solutions (depicted in Figure \ref{fig1}).

\begin{figure}[ptbh]
\centering
\includegraphics[width=2.2in]{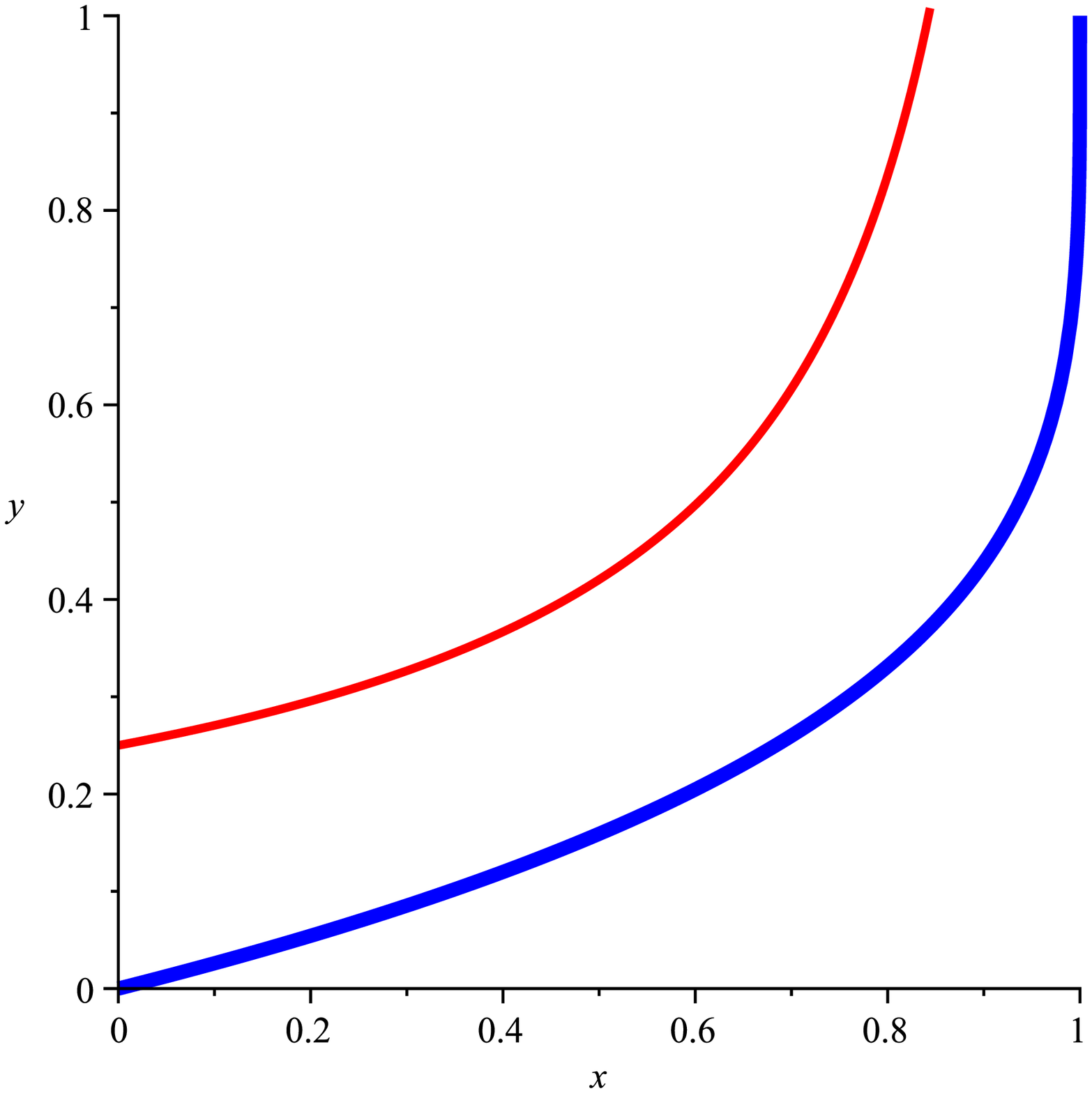}
\includegraphics[width=2.2in]{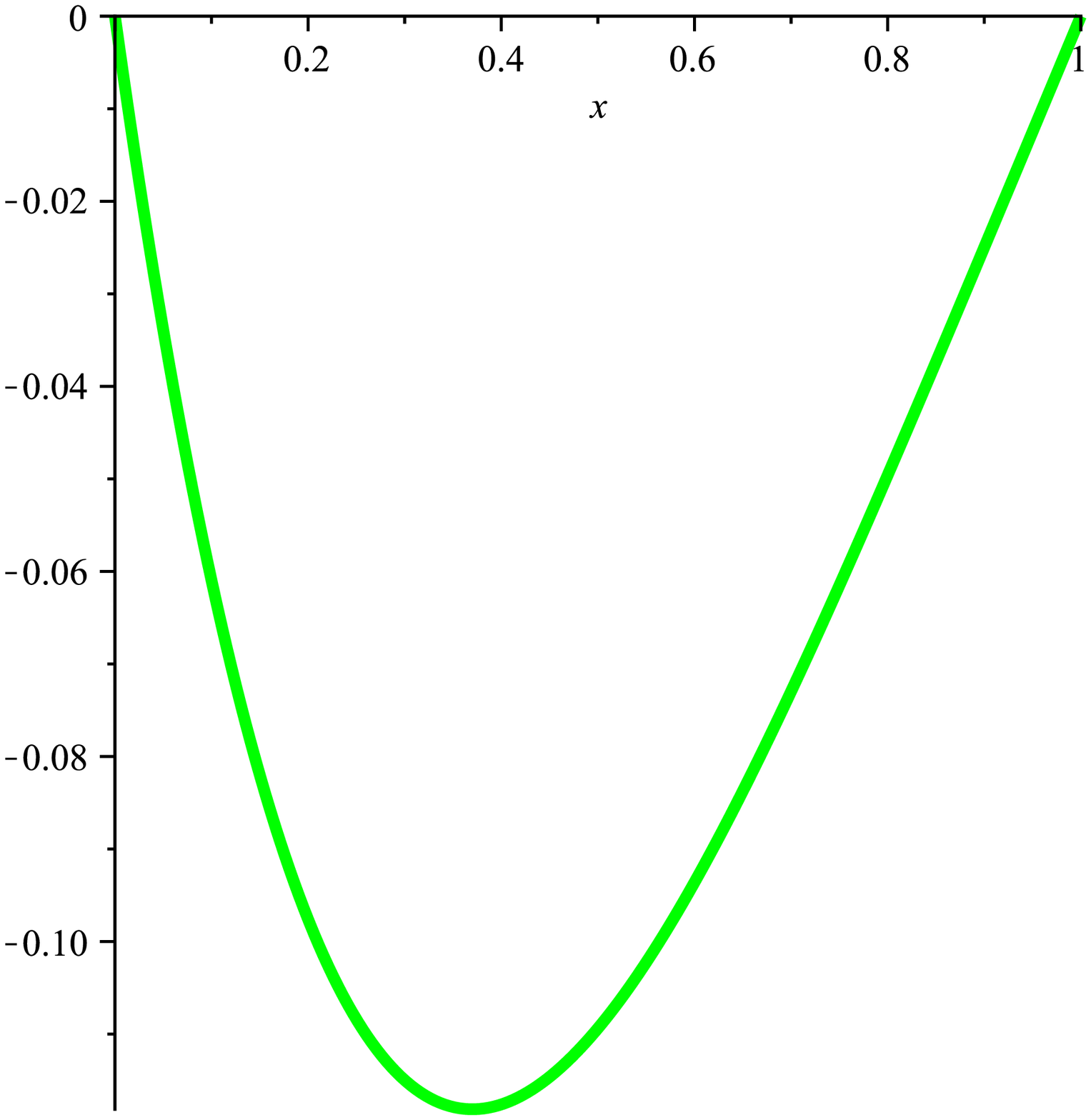}
\caption{\emph{In (\ref{eq:defexo}) and (\ref{eq:defendo}) we use $n=4$, $\alpha_1=1$ and $\alpha_i=0$ for $i>1$. The first figure shows the solution $\Gamma^*_{exo}(N)$ given in Equation (\ref{eq:explicit}) (thick blue curve) and its density (thinner red curve) for exogenous evolution. In the second figure we plot $\phi_{endo}(R)-R$ to indicate where the fixed points of $\phi_{endo}$ are located. The limiting measure is concentrated on the zero of this function with positive slope, ie: $R=1$.}}
\label{fig1}
\end{figure}

\begin{figure}[ptbh]
\centering
\includegraphics[height=2.2in]{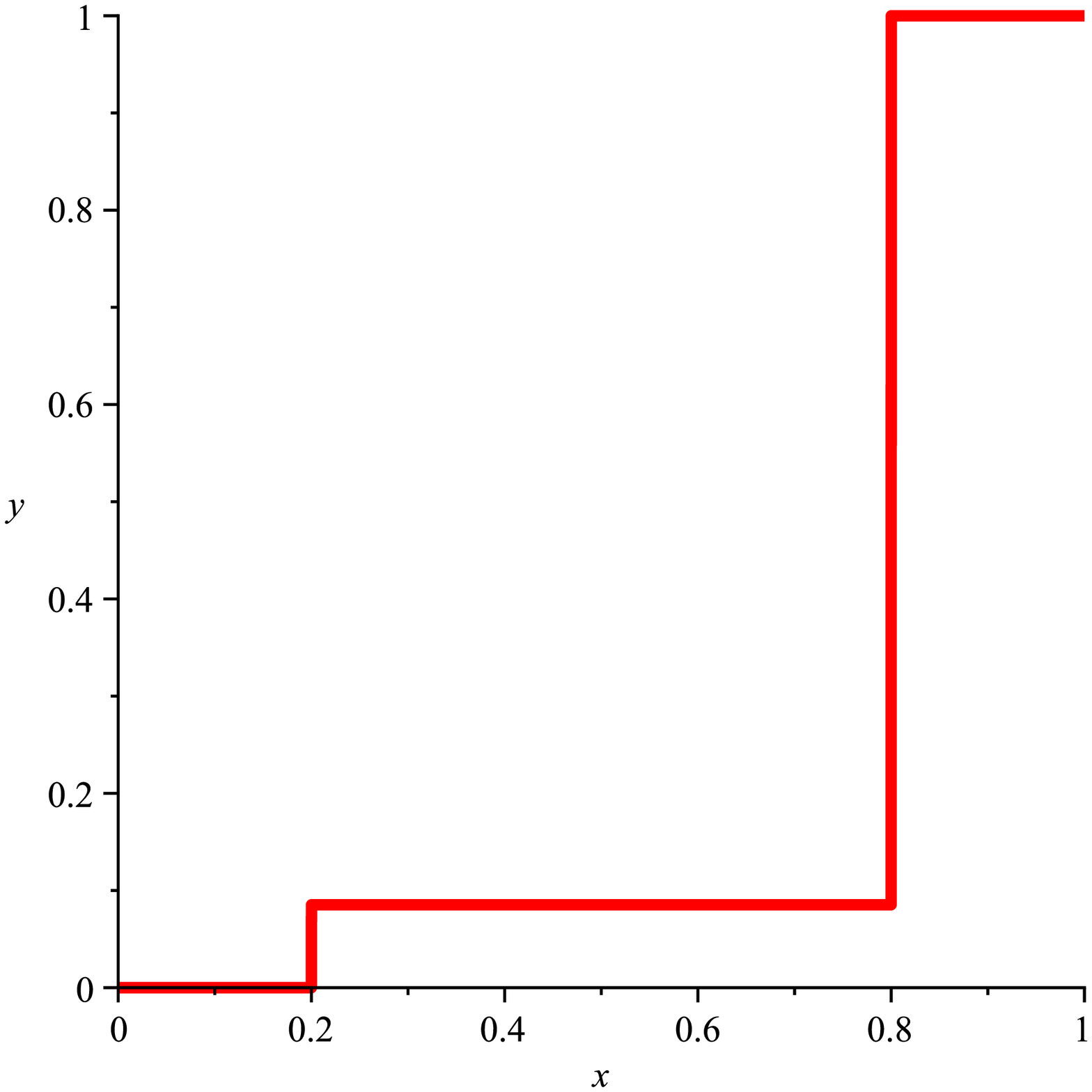}
\includegraphics[height=2.2in]{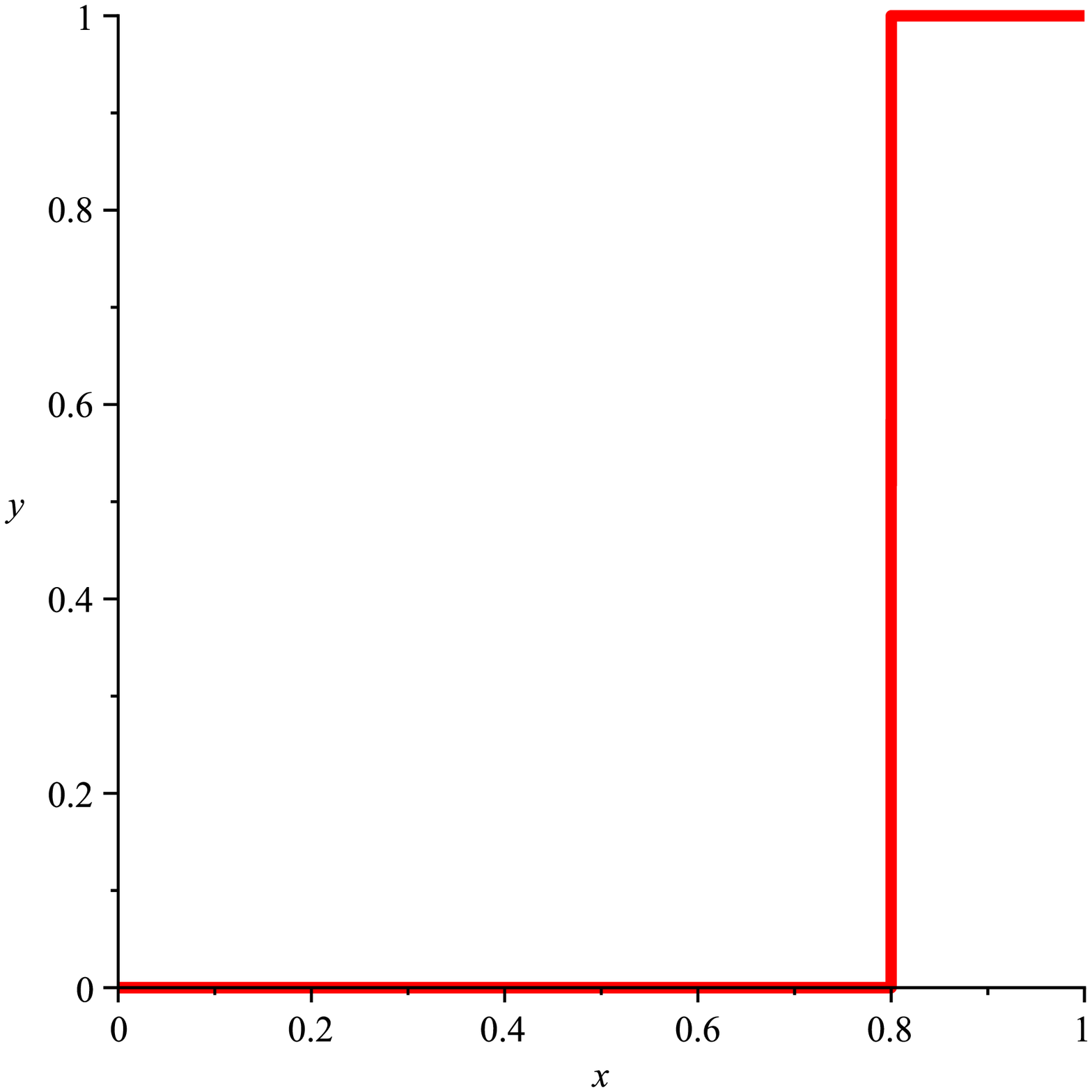}
\caption{\emph{For the same problem as Figure \ref{fig1} we choose particular initial conditions, namely $N$ and $R_0$ are given by the unfair coin-toss in (\ref{eq:unfaircoin}). We plotted the limiting measure $\Gamma_{exo}^*(N(x))$ in the first picture and $\Gamma_{endo}^*(R_0(x))$ in the second.}}
\label{fig:coin-toss}
\end{figure}

\begin{figure}[ptbh]
\centering
\includegraphics[width=2.2in]{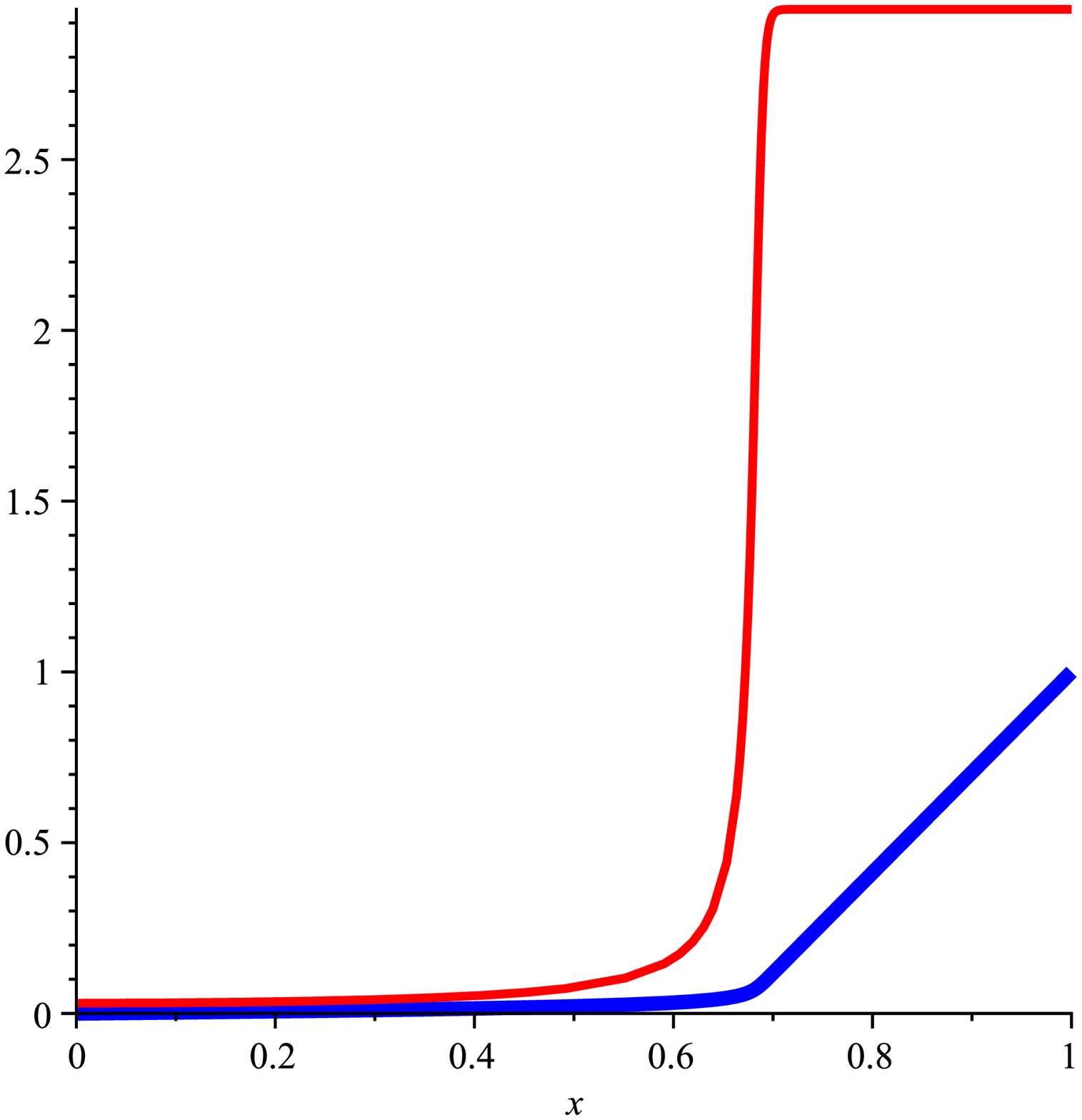}
\includegraphics[width=2.2in]{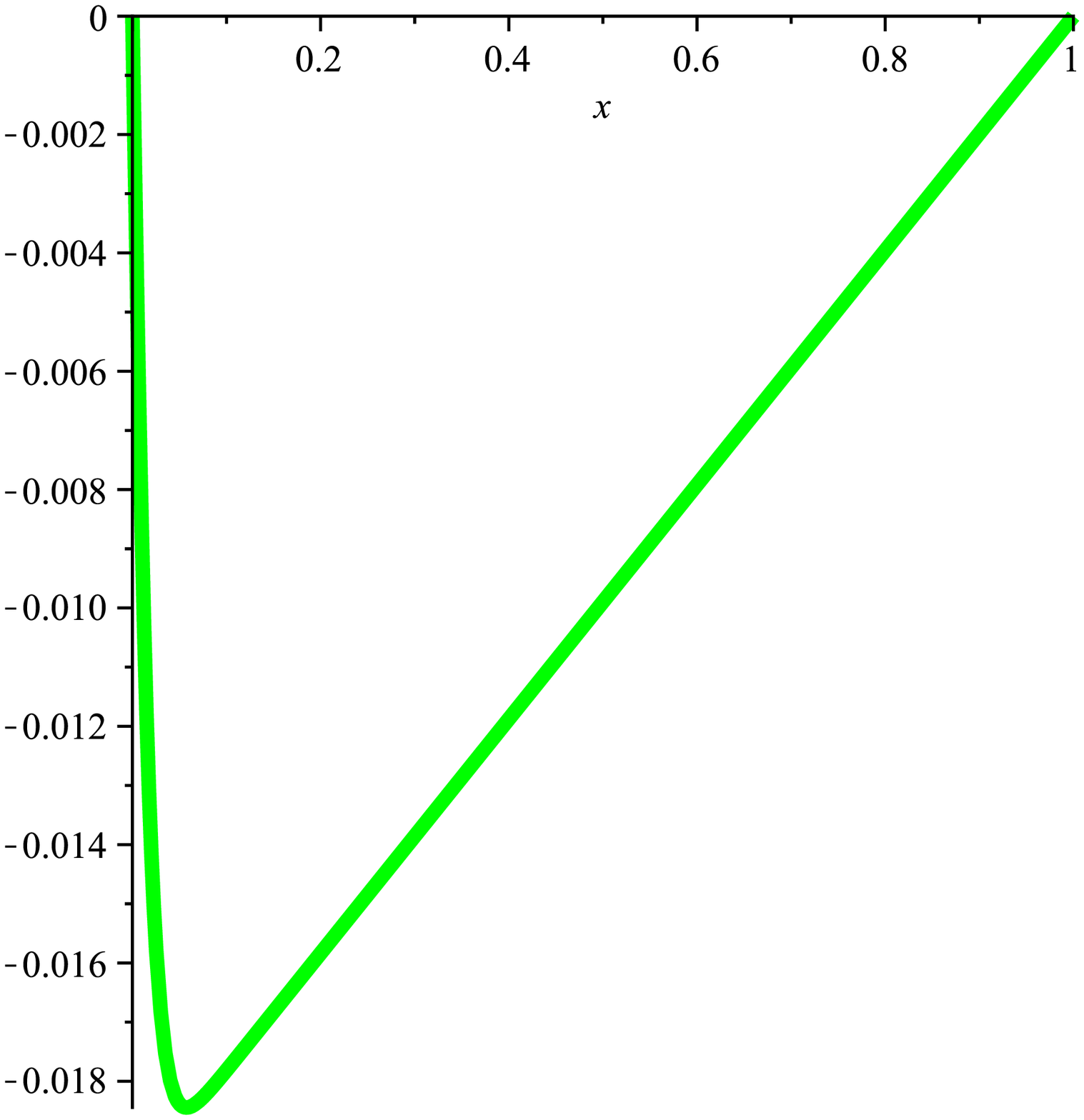}
\caption{\emph{In (\ref{eq:defexo}) and (\ref{eq:defendo}) we use $n=100$, $\alpha_1=\alpha_2=1$ and $\alpha_i=1/(n-2)$ for $i>2$. The first figure is the solution $\Gamma^*_{exo}$ (in red) from (\ref{eq:exo}) and (in blue) its density. In the second figure we draw $\phi_{endo} (R)-R$ to indicate the fixed points.}}
\label{fig5}
\end{figure}

\begin{figure}[ptbh]
\centering
\includegraphics[width=2.2in]{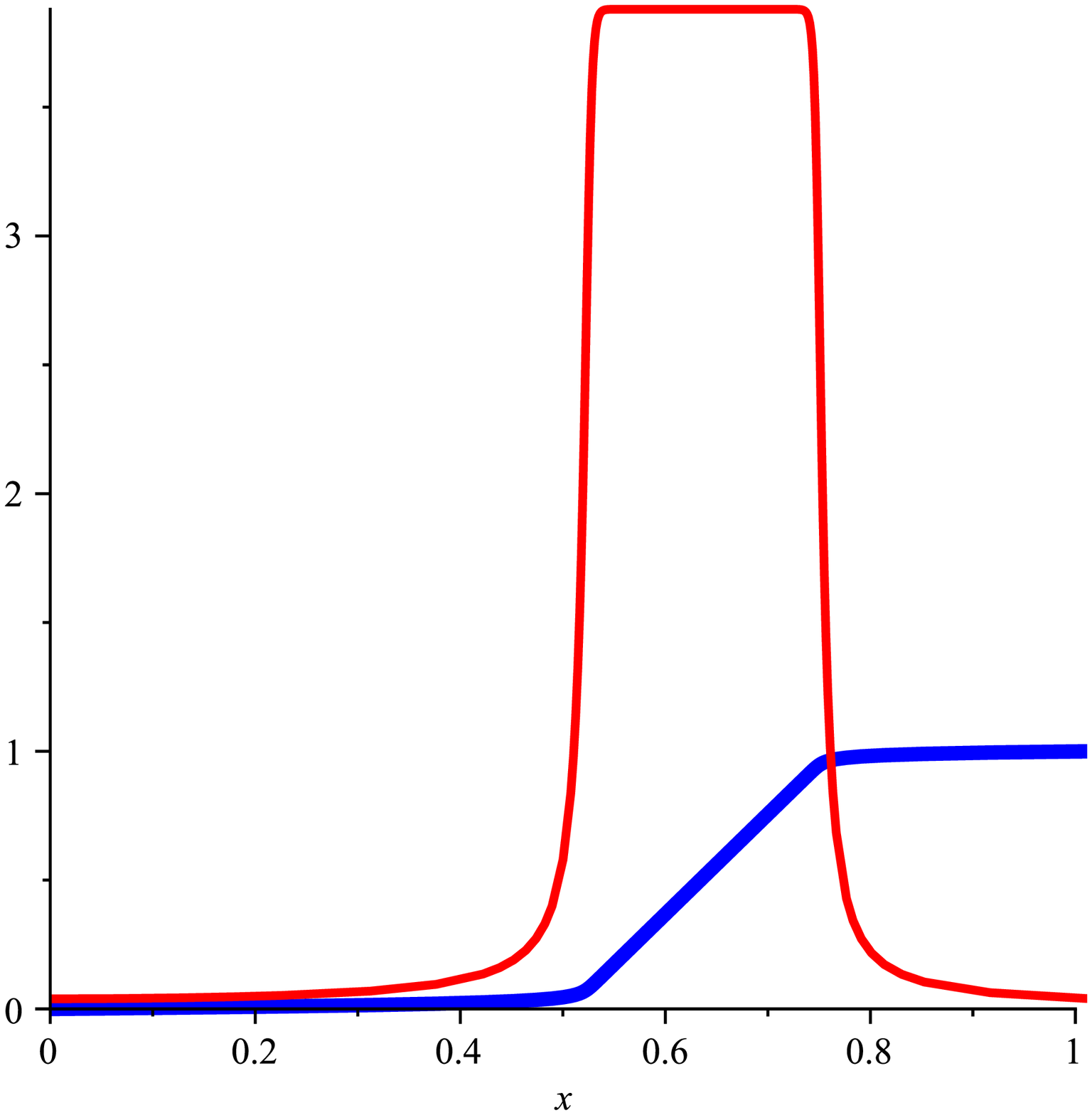}
\includegraphics[width=2.2in]{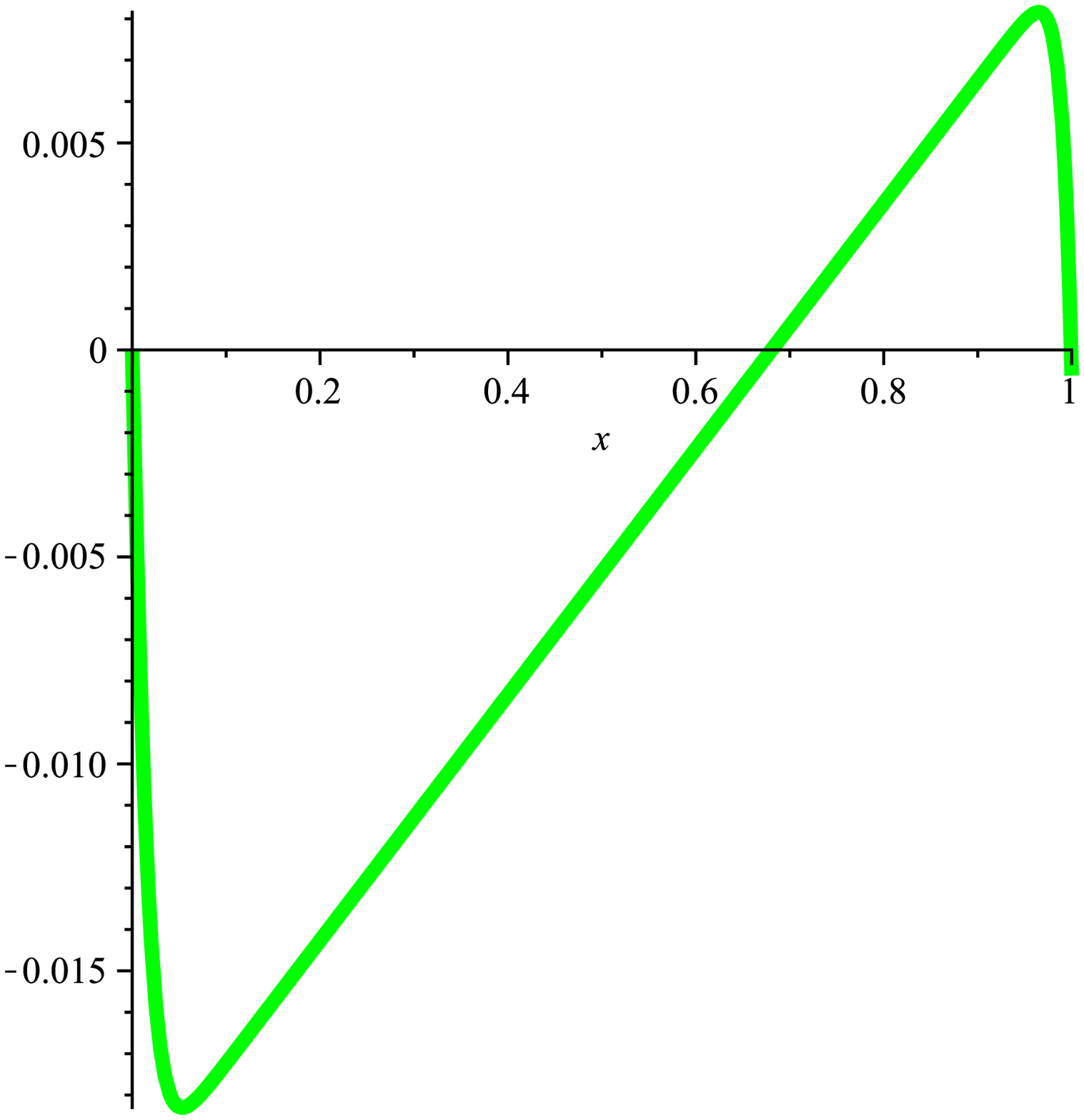}
\caption{\emph{In (\ref{eq:defexo}) and (\ref{eq:defendo}) we use $n=100$, $\alpha_1=\alpha_2=\alpha_n=1$ and $\alpha_i=1/(n-3)$ otherwise. The first figure is the solution $\Gamma^*_{exo}$ (in red) from (\ref{eq:exo}) and (in blue) its density. In the second figure we draw $\phi_{endo} (R)-R$ to indicate the fixed points.}}
\label{fig7}
\end{figure}

\begin{figure}[ptbh]
\centering
\includegraphics[width=2.2in]{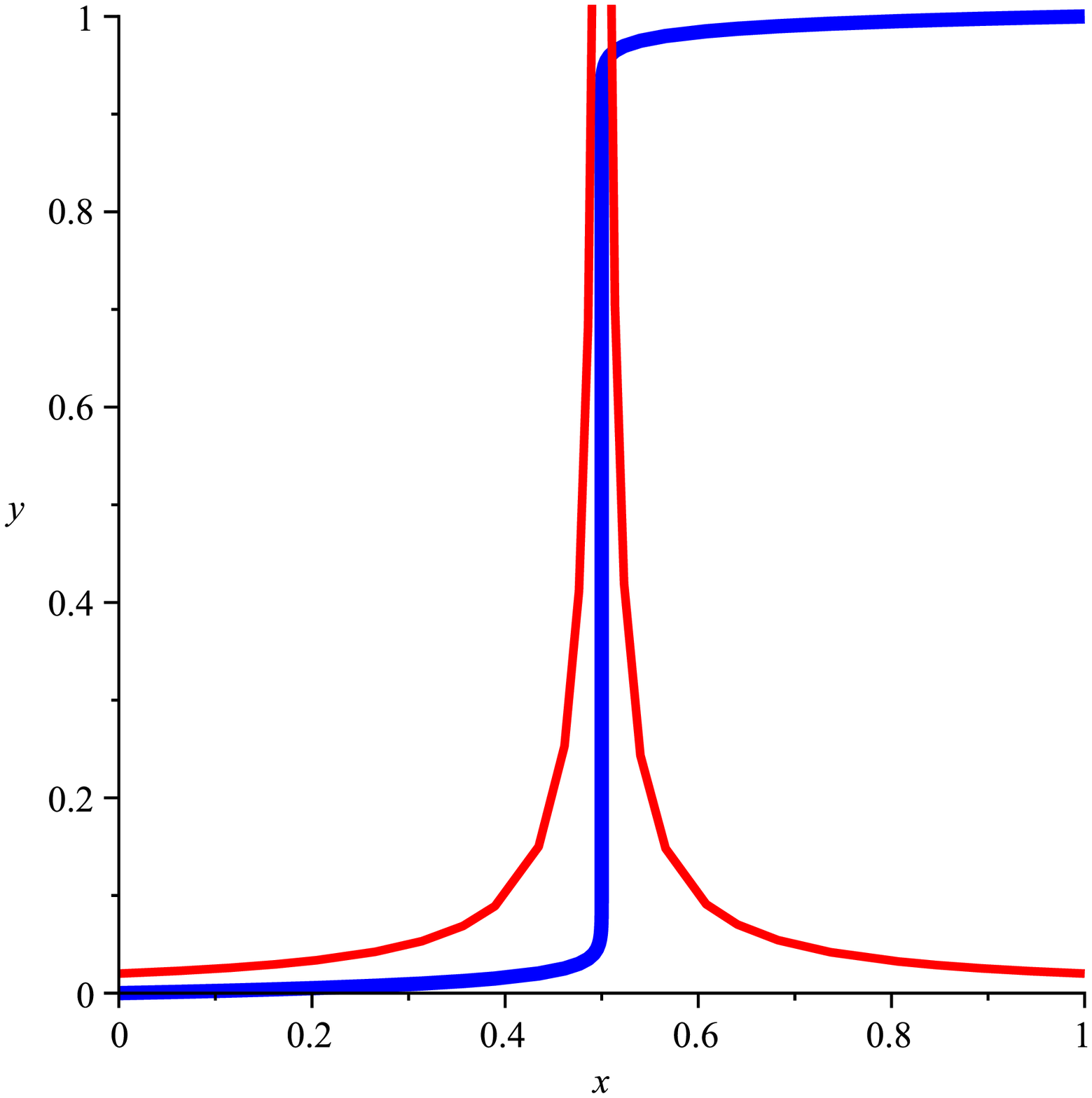}
\includegraphics[width=2.2in]{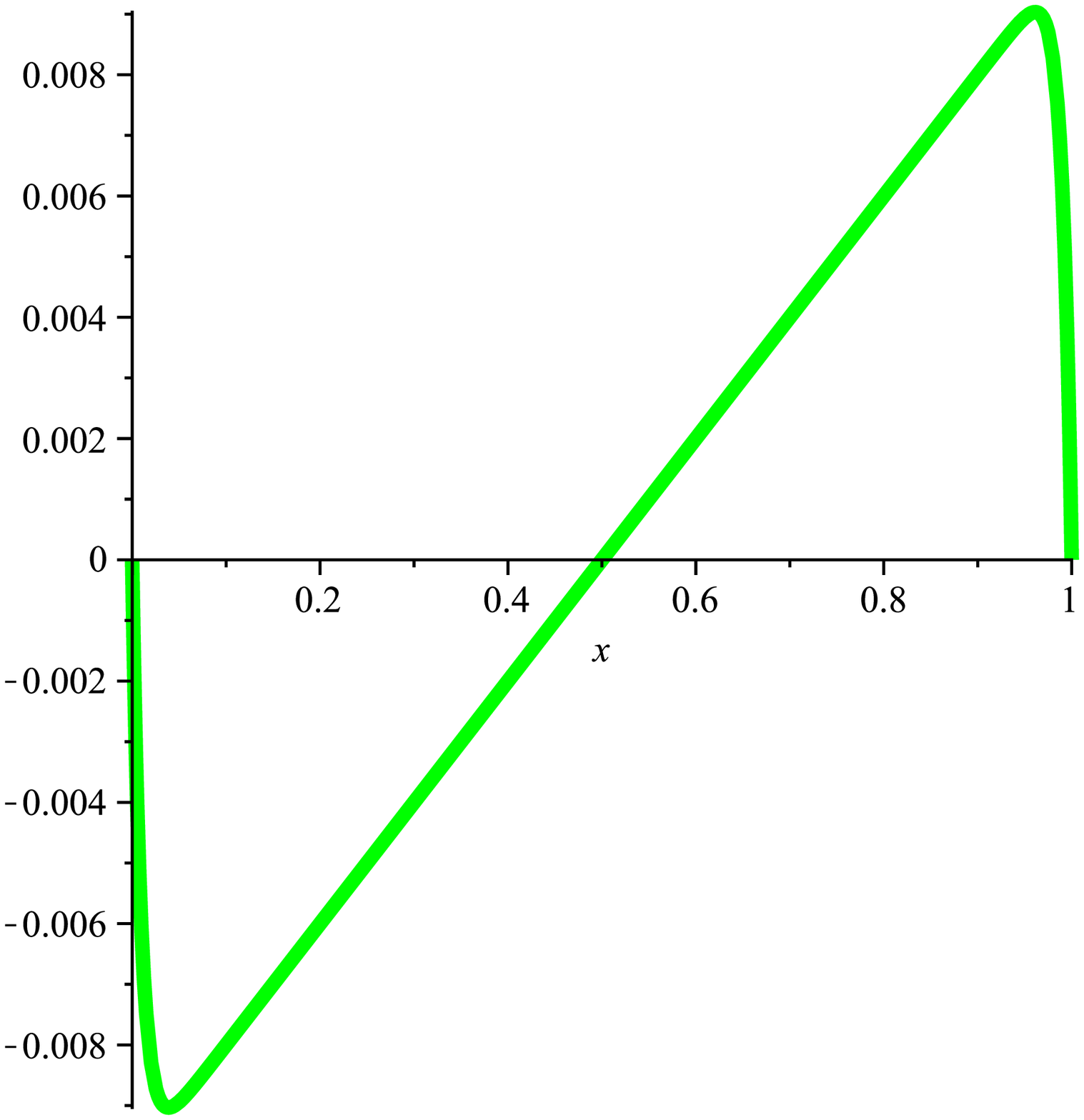}
\caption{\emph{In (\ref{eq:defexo}) and (\ref{eq:defendo}) we use $n=100$, $\alpha_1=\alpha_n=1$ and $\alpha_i=0$ otherwise. The first figure is the solution $\Gamma^*_{exo}$ (in red) from (\ref{eq:exo}) and (in blue) its density. In the second figure we draw $\phi_{endo} (R)-R$ to indicate the fixed points.}}
\label{fig6}
\end{figure}

The output depends on the initial conditions, however. For example, suppose in the same problem that we just discussed, we now have initial conditions corresponding to an unfair coin-toss, that is:
\begin{equation}
  N(x) = R_0(x) = pH_\alpha+qH_\beta = 0.3H_{0.2}+0.7H_{0.8} .
\label{eq:unfaircoin}
\end{equation}
Here, $0<\alpha<\beta<1$ and $p+q$ = 1. It can easily be checked analytically that we obtain for $\Gamma_{exo}^*(N(x))$ and $\Gamma_{endo}^*(N(x))$ the distributions of Figure \ref{fig:coin-toss}. Notice that the limiting distribution in the exogenous case is singular because the distribution $N(x)$ is singular (see Theorem \ref{theo:theo1}).

A few more examples are presented in Figures \ref{fig5}, \ref{fig6}, and \ref{fig7}. In these figures we draw $\Gamma_{exo}^*(N)$ (and its derivative) and $\phi_{endo}(R)-R$. The parameter values are specified in the Figures. In Figure \ref{fig5} we chose the weights $\alpha_i$ \emph{qualitatively} similar to what has been observed numerically in BS. We analyze this situation in detail in the next section. In Figure \ref{fig7} both low and high numbers are replaced giving rise to the distribution given there. The curious \emph{flat part} of the limiting distribution actually arises by virtue of replacing \emph{random} numbers in the distribution. In Figure \ref{fig6} we did not replace ``random'' numbers but only the lowest and the highest ranked numbers. The resulting exogenous distribution looks singular but isn't (by Theorem \ref{theo:theo1}). However one can show that the maximum of the density grows exponentially in $n$.

\section{Asymptotic behavior of exogenous rank-driven models}
\label{chap:asympt}

In this Section we show that our exogenous model can mimic the observed limiting (large number of agents) behavior of the distribution in the Bak-Sneppen model. To give an idea of the results of this section, consider as an example the model with $n > 2$, $\alpha_1=\alpha_2=1$, $\alpha_i=1/(n-2)$, for $i = 3,\ldots,n$ and $N(x)=x$ that was discussed in the previous section and exhibited in the first plot of Figure \ref{fig5}. This model has a behavior similar to that of the Bak-Sneppen model. Proposition \ref{prop:boundaries} below implies that
\begin{equation*}
  \partial_x|_{N=0}\,\Gamma^*\circ N(x)=1 \quad \mbox{ and }\quad
  \partial_x|_{N=1}\,\Gamma^*\circ N(x)=n^{-1}
\end{equation*}
The main theorem in this section (Theorem \ref{theo:theo2}) implies that if we fix $x<2/3$, then as $n$ tends to infinity $\Gamma^*\circ N(x)$ tends to 0. However if we fix $x>2/3$ then as $n$ tends to infinity $\Gamma^*\circ N(x)$ tends to $3x-2$.

We start with a proposition that characterizes the behavior of the fixed-point distribution as $N\rightarrow 0$ and $N\rightarrow 1$.

\begin{proposition}
Theorem \ref{theo:theo1} implies that the exogenous system of Definition \ref{defn:exo-endo} has a unique solution $\Gamma^*(N)$. We have:
\begin{equation*}
\partial_N|_{N=0}\,\Gamma^*(N)=\dfrac{q}{n\alpha_1} \quad \mbox{ and }\quad
\partial_N|_{N=1}\,\Gamma^*(N)=\dfrac{q}{n\alpha_n}
\end{equation*}
\label{prop:boundaries}
\end{proposition}

\begin{proof}
This follows directly from differentiation of Lemma \ref{lem:fixedpoints} using the chain rule:
\begin{equation*}
\partial_{R^*} \left(\sum_{i=1}^n\,\alpha_i \Theta_{i:n}(R^*)\right)\,\partial_N \Gamma^* (N) = \partial_N (qN) = q
\end{equation*}
and subsequent application of Lemma \ref{lem:derivatives}.
\end{proof}

To formulate the main theorem we need the following refinement of the model in Definition \ref{defn:exo-endo}. For the remainder of the Section we will be considering a sequence of models for increasing $n$, and we will be particularly interested in the case $n \rightarrow \infty$. In particular, $\alpha_i \equiv \alpha_i (n)$ in all that follows.

\begin{definition}
Let $\alpha_i$ as introduced in Definition \ref{defn:exo-endo}. We will say that a sequence (indexed by $n$) of sequences $\{ \alpha_i(n) \}_{i=1}^n$ is \emph{a $(p,q)$-sequence}, if there exist $p,q > 0$ independent of $n$, and $k(n)$, $\tilde \alpha(n)$, and $\delta_i(n)$ for $i > k(n)$ with the following properties:
\begin{enumerate}
  \item $p\geq\sum_{i=1}^{k(n)}\,\alpha_i$, $p < \sum_{i=1}^{k(n)+1}\,\alpha_i$ and $q=\sum_{i=1}^n\,\alpha_i$ and $\forall\; i>k(n)\;:\; \alpha_i=\tilde \alpha (n) + \delta_i (n)$
  \item $1 < k(n) < n$ and $k(n) = o(n/\ln n)$
  \item $\max_{i > k(n)} | \delta_i (n) | = o(1/n)$
\end{enumerate}
\label{def:pandq}
\end{definition}

\noindent
To illustrate this definition, consider the following examples of $(p,q)$-sequences: a) $\alpha_i = 2/\sqrt{n}$ for $i\leq \sqrt{n}$ and $\alpha_i = 1/(n-\sqrt{n})$ for $i > \sqrt{n}$, with $k(n) = \sqrt{n}$ (and $p = 2$); b) $\alpha_1 = \alpha_2 = 1$ and $\alpha_i = 1/(n-2)$ for $i \geq 3$, with $k(n) = 2$ (and $p = 2$). Alternatively, the values $\alpha_i = 3/n$ would define a sequence that is not $(p,q)$.

For $(p,q)$-sequences we can write
\begin{equation}
\label{eq:kovern}
   \alpha_i = \frac{q-p}{n-k(n)} + o(1/n) , \quad  \tilde p \equiv \sum_{i=1}^{k(n)} \alpha_i , \quad \tilde \alpha = \frac{q-p}{n-k(n)} .
\end{equation}

An important observation is that if a sequence is $(p,q)$, then the values
$p$ and $q$ are unique. The statement for $q$ is trivial. Assume that there exist $p_1$, $p_2$ with $p_2 \geq p_1$ such that the conditions in Definition 2 are satisfied for both values. We must have functions $k_1(n) \leq k_2(n)$ associated to $p_1$ and $p_2$ respectively, satisfying these conditions. Consider the sum
\[
    S(n) \equiv \sum_{i=k_1(n)+1}^{k_2(n)} \alpha_i > p_2 - p_1 \geq 0 .
\]
From Definition 2 it follows that
\begin{eqnarray*}
    \sum_{i=k_1(n)+1}^{k_2(n)} \alpha_i & = & \left( k_2(n) - k_1(n) - 1 \right) \tilde \alpha (n) + \sum_{i=k_1(n)+1}^{k_2(n)} \delta_i (n) \\
    & \leq & \left( k_2(n) - k_1(n) - 1 \right) \left( \frac{q-p}{n-k(n)} + \max_{i > k(n)} | \delta_i (n) | \right) = o\left( \frac{1}{\ln n} \right) ,
\end{eqnarray*}
This implies $S(n) \rightarrow 0$ and $p_2 = p_1$.

\vskip .1in
To prove the main Theorem we will need a technical Lemma. This result makes use of a function
$m(n):\mathbb{N} \rightarrow \mathbb{N}$ defined as
\begin{equation}
\label{eq:defm}
  m(n) \equiv \lceil k(n) \ln n \rceil.
\end{equation}

If Definition 2 holds, then for $\epsilon \in (0,1)$ and $n$ large enough we have
\begin{eqnarray}
\label{eq:defyc}
   m(n) & = & o(n) , \\
\label{eq:defya}
  \lim_{n\rightarrow \infty} m(n) & = & \infty , \\
\label{eq:defyb}
  m(n) & > & (k(n) + \epsilon ) \ln m(n) .
\end{eqnarray}

\begin{lemma}
Suppose $m(n)$ and $k(n)$ satisfy (\ref{eq:defm}), (\ref{eq:defyc}), (\ref{eq:defya}),
and (\ref{eq:defyb}). For $j \leq k(n)$ define
\begin{equation*}
  P_{j,n} (m) \equiv  \sum_{i=0}^{j-1}\,{n \choose i} \left(\frac{m(n)}{n} \right)^i
  \left(1- \frac{m(n)}{n}\right)^{n-i}
\end{equation*}
Then
\begin{equation*}
  \lim_{n\rightarrow\infty}\,\max_{1\leq j\leq k(n)}\,P_{j,n} (m) = 0
\end{equation*}
\label{lem:technical}
\end{lemma}

\begin{proof}

We have that $P_{j,n}\leq P_{k,n}$ and $P_{k,n}$ satisfies
\begin{equation*}
  P_{k,n} \leq \left(1-\frac{m}{n}\right)^{n} \sum_{i=0}^{k-1} \dfrac{1}{i!}
  \left( \dfrac{m}{1- \frac{m}{n}} \right)^i .
\end{equation*}
By (\ref{eq:defyc}) we have that $n > 2m$, and so $i!\left(1 - m/n \right)^i > 1/2$. Also for $m > 2$
and $k > 1$, we have that $\sum_{i=0}^{k-1}\,m^i\leq m^k$. Thus
\begin{equation*}
  P_{k,n} \leq 2\left(1-\frac{m}{n}\right)^n m^k ,
\end{equation*}
and
\begin{equation*}
  \ln P_{k,n} \leq \ln 2 + n\ln\left(1-\frac{m}{n} \right) + k\ln m .
\end{equation*}
Using $\ln(1-x) < -x$ for $x \in (0,1)$, and (\ref{eq:defyb}):
\begin{equation*}
  \ln P_{k,n}\leq \ln 2 - \epsilon \ln m .
\end{equation*}
From (\ref{eq:defya}) we have $\ln P_{k,n} \rightarrow -\infty$.
\end{proof}

\begin{theorem}
Theorem \ref{theo:theo1} implies that the exogenous system of Definition \ref{defn:exo-endo} has a unique solution $R_n^*= \Gamma^*(N;n)$ for each $n$. If $\alpha_i(n)$ is a $(p,q)$-sequence, these solutions satisfy:
\begin{enumerate}
  \item $N(x) < \dfrac pq\quad \Rightarrow \quad \displaystyle\lim_{n\rightarrow \infty} R^*_n (x) = 0$
  \item $N(x) > \dfrac pq\quad \Rightarrow \quad \displaystyle\lim_{n\rightarrow \infty} R^*_n (x) = \dfrac{q N(x) - p}{q - p}$
\end{enumerate}
\label{theo:theo2}
\end{theorem}

\begin{proof}
$R_n^*(y)=\Gamma^*(N(y);n)$ are non-decreasing functions of $N$ by Theorem \ref{theo:theo1}.
We prove part 1 by finding values $y_n$ tending to $p/q$ such that $R_n^*(y_n)$ tends to 0.

Using $\alpha_i$, $k(n)$ and $q$ given in Definition 2, and $m(n)$ defined in (\ref{eq:defm}), we can write the fixed-point equation from Lemma \ref{lem:fixedpoints} as:
\begin{equation}
\label{eq:eq1}
  qN = \sum_{i=1}^{k}\,\alpha_i \Theta_{i:n} - \tilde \alpha \sum_{i=1}^{k}\,\Theta_{i:n} + \tilde \alpha \sum_{i=1}^n\, \Theta_{i:n} + \sum_{i=k+1}^n\, \delta_i \Theta_{i:n},
\end{equation}
where we have suppressed the dependence of $\alpha_i$ and $\tilde \alpha$ on $n$. Let $y_n \in [0,1]$ be such that $R^*_n \equiv R^*(y_n) = m(n)/n$, and use Lemma \ref{lem:order-stats}, the notation introduced in (\ref{eq:kovern}), and Lemma \ref{lem:technical} to obtain:
\begin{equation*}
  qN_n = \sum_{i=1}^{k}\,\alpha_i(1-P_{i,n}) -\frac{q-p}{n-k} \sum_{i=1}^{k}\,\Theta_{i:n}(R^*_n) + n\frac{q-p}{n-k} R_n^* + \sum_{i=k+1}^n\, \delta_i \Theta_{i:n}(R^*_n) ,
\end{equation*}
where $N_n \equiv N(y_n)$.

From (\ref{eq:kovern}), (\ref{eq:defyc}), and parts 2 and 3 of Definition 2 we have for $n \rightarrow \infty$,
\begin{eqnarray}
\label{eq:eq2}
   0 & \leq & \frac{q-p}{n-k} \sum_{i=1}^{k}\,\Theta_{i:n}(R^*_n) \leq \frac{q-p}{n-k} k \rightarrow 0 , \nonumber \\
   0 & \leq & n \frac{q-p}{n-k} R_n^* \leq n \frac{q-p}{n-k} \frac{m}{n} \rightarrow 0 , \\
  0 & \leq & \sum_{i=k+1}^n\, \delta_i \Theta_{i:n} (R^*_n) \leq (n-k-1) \max_{i>k} | \delta_i | \rightarrow 0 \nonumber .
\end{eqnarray}
Furthermore, we note that by Lemma \ref{lem:technical}
\[
   0 \leq \sum_{i=1}^{k(n)} \alpha_i P_{i,n} \leq \max_{i\leq k(n)} P_{i,n} \sum_{i=1}^{k(n)} \alpha_i \leq p \max_{i\leq k(n)} P_{i,n} \rightarrow 0 .
\]
Thus the first term tends to $p$ and we obtain
\begin{equation*}
  \lim_{n\rightarrow\infty}\,qN_n = p .
\end{equation*}
which proves the first part.

We prove part 2 by establishing that if $y$ is such that $R_n^*(y) = \Gamma^*(N (y);n) > \epsilon$ (independent of $n$), then $N(y) > p/q$ and $R_n^*(y)$ must have the indicated limit.

In the proof of part 1 we showed that whenever $R_n^*(y)=m(n)/n$, then $\lim_{n\rightarrow\infty}\,\sum_{i=1}^{k(n)}\,\alpha_i \Theta_{i:n}(R_n^*(y)) = p$. Since the $\Theta_{i:n} (R_n^*)$ are non-decreasing functions of $R_n^*$ (see Lemma \ref{lem:derivatives}), for $R_n^*(y) \geq\epsilon>0$ we must have that
\begin{equation*}
   \lim_{n\rightarrow\infty}\,\sum_{i=1}^{k(n)}\,\alpha_i \Theta_{i:n} (R_n^* (y)) \geq p .
\end{equation*}
On the other hand, from Theorem \ref{theo:order-statistics} and for any value $y$,
\[
  \sum_{i=1}^{k(n)}\,\alpha_i \Theta_{i:n} (R_n^*(y)) = \sum_{i=1}^{k(n)}\,\alpha_i - \sum_{i=1}^{k(n)}\,\sum_{j=0}^{i-1}\,{n \choose j} (R_n^*(y))^j (1- R_n^*(y))^{n-j} \leq p .
\]
Thus,
\[
  \lim_{n\rightarrow \infty} \sum_{i=1}^{k(n)} \alpha_i \Theta_{i:n} (R_n^*(y)) = p .
\]

Combining these bounds and using the arguments in (\ref{eq:eq2}) to have the second and fourth terms of (\ref{eq:eq1}) again limiting to 0, we get for any $\epsilon > 0$ and $R_n^*(y_n) \geq \epsilon$, letting $N_n \equiv N(y_n)$,
\begin{equation*}
  \lim_{n\rightarrow\infty} \left( qN_n - p - \frac{q-p}{1-k(n)/n} \Gamma^*(N_n;n) \right) = 0 ,
\end{equation*}
which proves part 2 of the theorem.
\end{proof}

This result is similar in spirit to Theorem 3.2 in \cite{GKW2}, although simpler in its structure. Also, it applies to any driving distribution $N$, and not just to the uniform case. It provides a very close approximation to the simulation results presented in \cite{FSB}, Fig.\ 1, and specially in \cite{GD}, Fig.\ 1, for the form of the stationary marginal densities in the Bak-Sneppen model.

The result can be extended to include other situations. For instance in Definition 2 instead of using a single number $p$ one could require
\begin{equation*}
  p_1=\sum_{i=1}^{k_1(n)}\,\alpha_i\quad \mbox{and}\quad p_2=\sum_{i=k_2(n)}^{n}\,\alpha_i .
\end{equation*}
The example of Figure \ref{fig7} is a case in point.

\section{Comparison of rank-driven models}
\label{chap:discussion}

We compare the performance of our approximating rank-driven dynamical model to a rank-driven model that we can analyze exactly. For the remainder of the section we consider the exogenous process with $n$ agents, where $\alpha_1=1$ and the remaining $\alpha_i$'s are all zero. The solution (without approximation) is described in the following proposition (see \cite{GKW2} for more details).

\begin{proposition}
For the exact exogenous system with $n$ agents, $\alpha_1=1$, and the remaining $\alpha_i$'s equal to zero, its distribution $\hat R(x)$ converges to $\hat{\Gamma}^* (N)(x)$ and:
\begin{equation*}
  \hat{\Gamma}^* (N) = \dfrac1n N + \dfrac{n-1}{n}\;H_\beta ,
\end{equation*}
where $N$ is the driving distribution, $H_\beta$ the Heaviside function with jump at $\beta$, and $\beta=\sup\{x\,|\,N(x)<1\}$.
\label{prop:exact}
\end{proposition}

\begin{proof} Let $\{x^{(i)}_t\}_{i=1}^{n}$ be the values of the fitnesses at time-step $t$, ordered according to increasing magnitude. At every positive integer $t$,  $x^{(1)}_t$ is replaced by a $N$-random value $x$. If $x<x^{(2)}_{t-1}$ then
\begin{equation*}
  x^{(1)}_t=x \quad \mbox{and} \quad x^{(i)}_t = x^{(i)}_{t-1} \quad {\rm for \;} i > 1 .
\end{equation*}
If that is not the case, then for some $k>1$ we have
\begin{equation*}
  x^{(k)}_t = x , \quad
  x^{(i)}_t = x^{(i+1)}_{t-1} \quad {\rm if \;} i<k \quad \mbox{and} \quad x^{(i)}_t = x^{(i)}_{t-1} \quad {\rm for \;} i > k .
\end{equation*}
It follows that for fixed $i>1$ the sequence $x^{(i)}_t$ (as function of $t$) is non-decreasing and bounded, so it has a limit. If it converges to a number less than $\beta$, where $N(\beta) < 1$, then with probability 1 the replacement value $x$ will at some point be larger than it. And then it must increase by the previous reasoning. Thus, must converge to a value $\geq\beta$.

In the limit the replacement value will always be $x^{(1)}_t$, the smallest entry in the list of fitnesses. Therefore the distribution of that variable must be equal to $N$. Thus, the order-statistics satisfy
\begin{equation*}
  \Theta_{1:n}(\hat R^*) = N \quad \mbox{ and } \quad
  \Theta_{i:n} (\hat R^*) = H_\alpha(x) \quad \mbox{for } i > 1.
\end{equation*}
The proposition follows from Lemma \ref{lem:order-stats}.
\end{proof}

If we introduce the independence assumption and approximate this model with a rank-driven dynamical system, then from Equation \ref{eq:explicit} we see that its limiting measure satisfies:
\begin{equation}
  \Gamma^*(N) = 1-(1-N)^{1/n}
\label{eq:limiting measure}
\end{equation}
We now provide a sense of the closeness of the distributions $R^*=\Gamma^*(N)$ and $\hat R^* = \hat\Gamma^* (N)$ by computing several derived quantities. From Proposition \ref{prop:exact} and Theorem \ref{theo:order-statistics} one easily sees that:

\begin{lemma}
\begin{equation*}
   \Theta_{1:n} (R^*(x)) = \Theta_{1:n} (\hat R^*(x)) = N(x)
\end{equation*}
\label{lem:means}
\end{lemma}

To simplify the following discussion we now set $N$ to be equal to the uniform distribution: $N(x)=x$.

\begin{corollary}
The mean of $\hat R^*(x)$ is $1-\dfrac{1}{2n}$. The mean of $R^*(x)$ is $1-\dfrac{1}{n+1}$.
\label{cory:means2}
\end{corollary}

\begin{proof} The first statement follows from Proposition \ref{prop:exact}. The second can be calculated using Equation \ref{eq:limiting measure} and
\begin{equation*}
  \mathbb{E} [ R^*] = \int_0^1 (1 - R^*(x) )\,dx = \int_0^1\,(1-x)^{1/n}\,dx = \dfrac{n}{n+1}
\end{equation*}
\end{proof}

The situation is slightly worse if we look at the distributions of individual order statistics. In this context, we would expect to see the largest difference between $\Theta_{2:n} (\hat R^*)$ and $\Theta_{2:n} (R^*)$.

\begin{corollary}
$\Theta_{2:n} (\hat R^*) = H_1$ and has mean 1, while $\Theta_{2:n} (R^*) (x) = 1+(n-1)(1-x)-n(1-x)^{(n-1)/n}$ and has mean $\dfrac{3n-1}{2(2n-1)}$.
\label{cory:means3}
\end{corollary}

\begin{proof} The first statement is obvious.

For the second statement, by Theorem \ref{theo:order-statistics} we get:
\begin{equation*}
  \Theta_{2:n} (R^*) = 1-(1-R^*)^n-nR^*(1-R^*)^{n-1}
\end{equation*}
Substituting Equation \ref{eq:limiting measure} gives the desired expression for $\Theta_{2:n} (R^*)$. For its mean, we have as before
\begin{eqnarray*}
  \mathbb{E} [ \Theta_{2:n} (R^*) ] & = & \int_0^1 (1 - \Theta_{2:n} (R^*)(x) )\,dx \\
  & = & n \int_0^1 (1-x)^{(n-1)/n} \,dx - (n-1) \int_0^1 (1-x) \,dx
\end{eqnarray*}
The evaluation of this integral yields the desired result.
\end{proof}

\section{Conclusions}
\label{chap:conclusions}

In this paper we have described and analyzed the asymptotic behavior of a class of
rank-driven processes. These processes appear in many practical cases, where a dynamic
system changes along time by replacing some of its components by others with new random characteristics, when the choice of the components to replace based on the ranks of these components. This would be a reasonable model for a system where the replacements come from a fixed distribution $N$ that is independent of the evolving distribution. This model was thought to apply to the ``biological evolution of interacting species'' by its original authors (see \cite{BS} for further discussion of this aspect). Other examples are for example the modeling of the dynamics of economic agents, see \cite{BLT} or \cite{ACP}, or the study of small-world networks, \cite{E}. As long as the components to be replaced are selected in an (approximately) independent manner, our theory shows that there is an essentially unique fixed point distribution which can easily be constructed from the initial distribution and the fixed distribution $N$.

In the paper we also show that the behavior of these systems is quite different if the replacements come from the evolving distribution itself. That would be the case if the characteristics of the replacements are somehow inherited from those of the existing population. In this case we have show that again there is a unique fixed point that can easily be constructed from the initial distribution. The difference with the previous case is that in this case the limiting distribution is always a countably singular measure, while in the former case there are many plausible scenarios where the limiting measure actually has a non-singular density.

\begin{acknowledgements}
The authors wish to thank Ra\'{u}l Jimenez for inspiring conversations, as well as the anonymous referees for their many suggestions to improve the presentation and the contents of the paper.

The work of one of the authors (F.J. Prieto) was partially supported by the Spanish Ministerio de Educacion, Cultura y Deporte grant MCYT ECO2011-25706.
\end{acknowledgements}


\begin{thebibliography}{}

\bibitem{ACP}
Ausloos, M., Clippe, P., Pekalski, A.: Evolution of economic entities under heterogeneous political/environmental conditions within a Bak–Sneppen-like dynamics. Physica A, 332, 394-402 (2004)

\bibitem{BS}
Bak, P., Sneppen, K.: Punctuated equilibrium and criticality in a simple model of evolution. Phys Rev Lett 71, 4083-4086 (1993)

\bibitem{BLT}
Bartolozzi, M., Leinweber, D.B., Thomas, A.W.: Symbiosis in the Bak–Sneppen model for biological evolution with economic applications. Physica A, 365, 499-508 (2006)

\bibitem{DN}
David, H.A., Nagaraja, H.N.: Order Statistics. Wiley Series in Probability and Statistics. Wiley, New Jersey (2003)

\bibitem{BDF}
de Boer, J., Derrida, B., Flyvbjerg, H., Jackson, A.D., Wettig, T.: Simple Model of Self Organized Biological Evolution. Phys Rev Lett 73, 906-909 (1994)

\bibitem{E}
Elettreby, M.F.: Multiobjective Bak–Sneppen model on a small-world network. Chaos, Solitons and Fractals, 26, 1009–1017 (2005)

\bibitem{FSB}
Flyvbjerg, H, Sneppen, K., Bak, P.: Mean Field Theory for a Simple Model of Evolution. Phys Rev Lett 71, 24, 4087-4090 (1993)

\bibitem{GD}
Garcia, G.J.M., Dickman, R.: On the thresholds, probability densities, and critical exponents of Bak–Sneppen-like models. Physica A 342, 164–170 (2004)


\bibitem{GKW2}
Grinfeld, M., Knight, P.A., Wade, A.R.: Rank-driven Markov Processes. J Stat Phys 146, 2, 378-407 (2012)

\bibitem{GKW1}
Grinfeld, M., Knight, P.A., Wade, A.R.: Bak-Sneppen-type Models and Rank-driven Processes. Phys Rev E 84, 041124 (2011)

\bibitem{Jen}
Jensen, H.J.: Self-Organized Criticality. Cambridge University Press, Cambridge (1998)

\bibitem{Mas}
Maslov, S.: Infinite Series of Exact Equations in the Bak-Sneppen Model of Biological Evolution. Phys Rev Lett 77, 1182-1185 (1996)

\bibitem{MZ1}
Meester, R., Znamenski, D.: Non-triviality of a Discrete Bak Sneppen Evolution Model. J Stat Phys 109, 987-1004 (2002)

\bibitem{MZ2}
Meester, R., Znamenski, D.: Limit Behavior of the Bak Sneppen Evolution Model. Ann Prob 31, 4, 1986-2002 (2003)

\bibitem{MZ3}
Meester, R., Znamenski, D.: Critical Thresholds and Limit Distribution in the Bak Sneppen Model. Comm Math Phys 246, 63-86 (2004)

\bibitem{MGW}
Meester, R., Gillett, A., Van Der Wal, P.: Maximal Avalanches in the Bak-Sneppen Model. J Appl Prob 43, 840-851 (2006)

\bibitem{MS}
Meester, R., Sarkar, A.: Rigorous Self-organised Criticality in the Modified Bak-Sneppen Model. J Stat Phys 149 964-968 (2012)

\bibitem{PMB}
Paczuski, M., Maslov, S., Bak, P.: Avalanche dynamics in evolution, growth, and depinning models. Phys Rev E 53, 414-443 (1996)

\bibitem{Tab}
Tabelow, K.: Gap function in the finite Bak-Sneppen model. Phys Rev E 63, 047101 (2001)

\end{thebibliography}
\end{document}